\documentclass{amsart}

\usepackage{latexsym,enumerate}
\usepackage{amsmath,amsthm,amsopn,amstext,amscd,amsfonts,amssymb}
\usepackage[ansinew]{inputenc}
\usepackage{verbatim}
\usepackage{graphicx}
\usepackage{pstricks}
\usepackage{wasysym}
\usepackage[all]{xy}
\usepackage{epsfig}

\usepackage{subfigure}
\usepackage{pstricks,pst-node}

\usepackage{multicol}
\usepackage{color}
\usepackage{colortbl}

\newtheorem{thm}{Theorem}[section]
\newtheorem{cor}[thm]{Corollary}
\newtheorem{lem}[thm]{Lemma}
\newtheorem{prop}[thm]{Proposition}
\newtheorem{defn}[thm]{Definition}

\DeclareMathOperator{\diam}{diam}

\DeclareMathOperator{\diameff}{effdiam}

\newcommand{\spb}[1]{\smallskip}
\newcommand{\mpb}[1]{\medskip}
\newcommand{\bpb}[1]{\bigskip}

\renewcommand{\d}{\delta}

\newcommand{\g}{\gamma}

\begin{document}
\DeclareGraphicsExtensions{.jpg,.pdf,.mps,.png}

\title[Extremal problems on graphs involving the hyperbolicity constant]{Several extremal problems on graphs involving the circumference, girth, and hyperbolicity constant}

\author[Ver\'onica Hern\'andez]{Ver\'onica Hern\'andez$^{(1)}$}
\address{Departamento de Matem\'aticas, Universidad Carlos III de Madrid,
Avenida de la Universidad 30, 28911 Legan\'es, Madrid, Spain}
\email{vehernan@math.uc3m.es}

\author[Domingo Pestana]{Domingo Pestana$^{(1)}$}
\address{Departamento de Matem\'aticas, Universidad Carlos III de Madrid,
Avenida de la Universidad 30, 28911 Legan\'es, Madrid, Spain}
\email{dompes\@@math.uc3m.es}

\author[Jos\'e M. Rodr{\'\i}guez]{Jos\'e M. Rodr{\'\i}guez$^{(1)}$}
\address{Departamento de Matem\'aticas, Universidad Carlos III de Madrid,
Avenida de la Universidad 30, 28911 Legan\'es, Madrid, Spain}
\email{jomaro@math.uc3m.es}

\thanks{$^{(1)}$ Supported in part by three grants from Ministerio de Econom{\'\i}a y Competititvidad, Agencia Estatal de
Investigación (AEI) and Fondo Europeo de Desarrollo Regional (FEDER) (MTM2013-46374-P, MTM2016-78227-C2-1-P and MTM2015-69323-REDT), Spain.}

\date{\today}

\maketitle{}

\begin{abstract}

To compute the hyperbolicity constant is an almost intractable problem, thus it is natural to try to bound it in terms of some parameters of the graph.
Let $\mathcal{G}(g,c,n)$ be the set of graphs $G$ with girth $g(G)=g$, circumference $c(G)=c$, and $n$ vertices; and let $\mathcal{H}(g,c,m)$ be the set of graphs with girth $g$, circumference $c$, and $m$ edges.
In this work, we study the four following extremal problems on graphs:
$A(g,c,n)=\min\{\delta(G)\,|\; G \in \mathcal{G}(g,c,n) \}$, $B(g,c,n)=\max\{\delta(G)\,|\; G \in \mathcal{G}(g,c,n) \}$, $\alpha(g,c,m)=\min\{\delta(G)\,|\; \in \mathcal{H}(g,c,m) \}$ and
$\beta(g,c,m)=\max\{\delta(G)\,|\; G \in \mathcal{H}(g,c,m) \}$.
In particular, we obtain bounds for $A(g,c,n)$ and $\alpha(g,c,m)$, and we compute the precise value of $B(g,c,n)$ and $\beta(g,c,m)$ for all values of $g$, $c$, $n$ and $m$.

\end{abstract}



\section{Introduction}

Gromov hyperbolicity was introduced by the Russian mathematician Mikhail Leonidovich Gromov in the setting of geometric group theory \cite{G1}, \cite{GH}, \cite{CDP}, but has played an increasing role in analysis on general metric spaces \cite{BHK}, \cite{BS}, \cite{BB}, with applications to the Martin boundary, invariant metrics in several
complex variables \cite{BBonk} and extendability of Lipschitz mappings \cite{La}.
The concept of hyperbolicity appears also in discrete mathematics, algorithms
and networking. For example, it has been shown empirically
in \cite{ShTa} that the internet topology embeds with better accuracy
into a hyperbolic space than into a Euclidean space
of comparable dimension (formal proofs that the distortion is related to the hyperbolicity can be found in \cite{VeSu});
furthermore, it is evidenced that many real networks are hyperbolic (see, e.g., \cite{CoCoLa,KPKVB,MoSoVi}).
Another important application of these spaces is the study of the spread of viruses through the internet (see \cite{K21,K22}).
Furthermore, hyperbolic spaces are useful in secure transmission of information on the
network (see \cite{K21,K22}).
In \cite{KNS} the authors study hyperbolicity in large scale networks (such as communication, citation, collaboration, peer-to-peer, friendship and other social networks) and propose that hyperbolicity, in conjunction with other local characteristics of networks, such as the degree distribution and clustering coefficients, provide a more complete unifying picture of networks, and helps classify in a parsimonious way what is otherwise a bewildering and complex array of features and characteristics specific to each natural and man-made network.
The hyperbolicity has also been used extensively in the context of random graphs (see, e.g., \cite{Sha1,Sha2,Sha3}).

The study of Gromov hyperbolic graphs is a subject of increasing interest in graph theory; see, e.g.,
\cite{BC,
BKM,
ChFaHuMa,CDEHV,CoCoLa,CD,CoDu,FIV,
K50,K21,K22,
KoMo,KPKVB,MP,
MoSoVi,
Sig,
WaJoBr,WZ}
and the references therein.

Last years several researchers have been interested in showing that metrics used in geometric function theory are Gromov hyperbolic.
In particular, the equivalence of the hyperbolicity of Riemannian manifolds and the hyperbolicity of a very simple graph was proved in \cite{PRT1,PT,T}, hence, it is useful to know hyperbolicity criteria for graphs.

Now, let us introduce the concept of Gromov hyperbolicity and the main results concerning this theory. For detailed expositions about Gromov hyperbolicity, see e.g. \cite{ABCD}, \cite{CDP}, \cite{GH} or \cite{Va}.

If $X$ is a metric space, we say that the curve $\g:[a,b]\longrightarrow X$ is a \emph{geodesic} if we have $L(\g|_{[t,s]})=d(\g(t),\g(s))=|t-s|$ for every $s,t\in [a,b]$
(then $\gamma$ is equipped with an arc-length parametrization).
The metric space $X$ is said to be \emph{geodesic} if for every couple of points in $X$ there exists a geodesic joining them; we denote by $[xy]$
any geodesic joining $x$ and $y$; this notation is ambiguous, since in general we do not have uniqueness of geodesics, but it is very convenient.
Consequently, any geodesic metric space is connected.
If the metric space $X$ is a graph, then the edge joining the vertices $u$ and $v$ will be denoted by $[u,v]$.

In order to consider a graph $G$ as a geodesic metric space, we identify (by an isometry) any edge $[u,v]\in E(G)$, where $E(G)$ denotes the edge set of $G$, with the interval $[0,1]$ in the real line;
then the edge $[u,v]$ (considered as a graph with just one edge)
is isometric to the interval $[0,1]$.
Thus, the points in $G$ are the vertices and, also, the points in the
interior
of any edge of $G$.
In this way, any graph $G$ has a natural distance defined on its points,
induced by taking the shortest paths in $G$,
and we can see $G$ as a metric graph.

If $X$ is a geodesic metric space and $x_1,x_2,x_3\in X$, a {\it geodesic triangle} $T=\{x_1,x_2,x_3\}$ is
the union of the three geodesics $[x_1x_2]$, $[x_2x_3]$ and
$[x_3x_1]$. We say that  $T$ is $\d$-{\it thin} if each of its sides is contained in the $\d$-neighborhood
of the union of the other sides. We denote by $\d(T)$ the sharp thin constant of $T$, i.e.,
$ \d(T)=\inf\{\d\ge 0| \, T \, \text{ is
$\d$-thin}\,\}.$ The space $X$ is $\d$-\emph{hyperbolic} if every geodesic triangle in $X$ is $\d$-thin. We denote by $\d(X)$ the sharp
hyperbolicity constant of $X$, i.e.,
$\d(X):=\sup\{\d(T)| \, T \,
\text{ is a geodesic triangle in }\,X\,\}.$ We say that $X$ is
\emph{hyperbolic} if $X$ is $\d$-hyperbolic for some $\d \ge 0$.

In the classical references on this subject (see, e.g., \cite{ABCD,GH}) appear
several different definitions of Gromov hyperbolicity, which are equivalent in
the sense that if $X$ is $\d$-hyperbolic with respect to one definition, then it is
$\d'$-hyperbolic with respect to another definition (for some $\d'$ related to $\d$).
We have chosen this definition because of its deep geometric meaning \cite{GH}.

The main examples of hyperbolic graphs are trees.
In fact, the hyperbolicity constant of a geodesic metric space can be viewed as a measure of
how ``tree-like'' the space is, since those spaces $X$ with $\delta(X) = 0$ are precisely the metric trees.
This is an interesting subject since, in many applications, one finds that the borderline between tractable and intractable cases may be the tree-like degree of the structure to be dealt with
(see, e.g., \cite{CYY}).
However, the hyperbolicity constant does not relate the graph in question to a specific tree (if connected) or forest (if not connected). In \cite{Sha5}, a measure called forest likelihood is introduced to estimate the likelihood of any given forest via a random dynamical generation process. This measure establishes an interesting connection between static graphs and dynamically growing graphs.

For a finite graph with $n$ vertices it is possible to compute $\d(G)$ in time $O(n^{3.69})$ \cite{FIV} (this is improved in \cite{CoCoLa,CD}).
Given a Cayley graph (of a presentation with solvable word problem) there is an algorithm which allows to decide if it is hyperbolic \cite{Pap}.
A refinement of this approach has been proposed in \cite{ChChPaPe}, that allows to do the same for many graphs: in particular, it provides a simple constant-factor approximation of the hyperbolicity constant of a graph on $n$ vertices in $O(n^2)$
time when the graph is given by its distance-matrix.
However, deciding whether or not a general infinite graph is hyperbolic is usually very difficult.
Therefore, it is interesting to relate hyperbolicity with other properties of graphs.
The papers \cite{BKM,WZ,BCRS,CRS2} prove, respectively, that chordal, $k$-chordal, edge-chordal and join graphs are hyperbolic.
Moreover, in \cite{BCRS} it is shown that hyperbolic graphs are path-chordal graphs.
These results relating chordality and hyperbolicity are improved in \cite{MP}.
Some other authors have obtained results on hyperbolicity for particular classes of graphs: vertex-symmetric graphs,
bipartite and intersection graphs, bridged graphs, expanders and median graphs
\cite{CaFu,CoDu,KoMo,LiTu,Sig}.

We consider simple (without loops or multiple edges) and connected  graphs such that every edge has length 1. Note that to exclude multiple edges and loops is not an important loss of generality, since \cite[Theorems 8 and 10]{BRSV2} reduce the problem of computing the hyperbolicity constant of graphs with multiple edges and/or loops to the study of simple graphs.
The vertex set of a graph $G$ is denoted by $V(G)$, and the \emph{order} $n$ of a graph is the number of its vertices ($n=|V(G)|$). The \emph{size} $m$ of a graph is the number of its edges ($m=|E(G)|$).

Throughout this work, by \emph{cycle} in a graph  we mean a simple closed curve, i.e., a path with different vertices, except for the last one, which is equal to the first vertex.

The \emph{circumference} of a graph (denoted by $c(G)$) is the length of any longest cycle in a graph, whereas  the  \emph{girth} of a graph (denoted by $g(G)$) is the length of any shortest cycle contained in the graph.

Along this paper $g$, $c$,  $n$ and $m$ are positive integers such that $3\leq g\leq c\leq  n   \leq  m$. Hence, we do not consider trees (note that $\delta(G)=0$ for every tree $G$).

Let $\mathcal{G}(g,c,n)$ be the set of graphs $G$ with   girth $g(G)=g$, circumference $c(G)=c$, and $n$  vertices;   and let $\mathcal{H}(g,c,m)$ be the set of graphs with girth $g$, circumference $c$, and $m$  edges.
Let us define
$$A(g,c,n):=\min\{\delta(G)\mid G \in \mathcal{G}(g,c,n) \},$$
$$B(g,c,n):=\max\{\delta(G)\mid G \in \mathcal{G}(g,c,n) \},$$
$$\alpha(g,c,m):=\min\{\delta(G)\mid G \in \\ \mathcal{H}(g,c,m) \},$$
$$\beta(g,c,m):=\max\{\delta(G)\mid G \in \mathcal{H}(g,c,m) \}.$$

Our aim in this paper is to estimate  $A(g,c,n)$, $B(g,c,n)$, $\alpha(g,c,m)$ and $\beta(g,c,m)$, i.e., to study the extremal problems of maximazing and minimazing $\delta(G)$ on the sets $\mathcal{G}(g,c,n)$  and $\mathcal{H}(g,c,m)$.\smallskip

The structure of this paper is as follows. In Section 2 we present key definitions, as well as previous results used in the paper. Sections 3 and 5 contain good bounds for $A(g,c,n)$ and $\alpha(g,c,m)$.
In Sections 4 and 6, Theorems \ref{12345} and \ref{12345e} give the precise value of $B(g,c,n)$ and $\beta(g,c,m)$ in any case, respectively.

\section{Previous results }

In order to estimate $A(g,c,n)$, $B(g,c,n)$, $\alpha(g,c,m)$ and $\beta(g,c,m)$, we need some previous results.

The following theorem gives lower and upper bounds for the hyperolicity constant of any graph in terms of its circumference and girth.  It is a direct consequence of
\cite [Theorem 17]{MRSV} and \cite [Lemma 2.11]{CRSV}.


\begin{thm} \label{t:1CRSV}
 For every graph $G$ with $g(G)=g$ and $c(G)=c$
  $$\dfrac{g}{4}\leq \delta(G)\leq \dfrac{c}{4},$$
 and both inequalities are sharp.
 \end{thm}

\begin{cor}
\label{newnew}
We always have
$$\dfrac{g}{4} \leq A(g,c,n)\leq B(g,c,n)\leq \dfrac{c}{4},$$
$$\dfrac{g}{4} \leq \alpha(g,c,m)\leq \beta(g,c,m)\leq \dfrac{c}{4}.$$
\end{cor}


Given a graph $G$ and $ [v,w  ]\in E(G)$, we say that $p$ is the \emph{midpoint} of $ [v,w  ]$ if $d_{G}(p,v)=d_{G}(p,w)=1/2$; let us denote by $J(G)$ the union of the set $V(G)$  and the midpoints of the edges of $G$. Consider the set $\mathbb{T}_1$ of geodesic triangles $T$ in $G$ that are cycles and such that the three vertices of the triangle $T$ belong to $J(G)$.

The following result states that in the hyperbolic graphs there always exists a geodesic triangle $T$ for which the hyperbolicity constant is attained and, furthermore, $T\in \mathbb{T}_1$. It appears in  \cite[Theorem 2.7]{BRS}.

\begin{thm}
\label{t:003} For any hyperbolic graph $G$
there exists a geodesic triangle $T\in \mathbb{T}_1$ such that $\delta(T)=\delta(G)$.
\end{thm}


Now we define a family of graphs which will be useful.

\begin{defn} \label{defn1}
Consider  non-negative integers $k$, $\beta_{j}$,  $\beta'_{j}$ ($0 \leq j\leq k$), and $\alpha_{j}$  ($0 \leq j\leq k+1$), with $\beta'_{0}= \beta'_{k}=\alpha_{0}=\alpha_{k+1}=0$, and such that
\begin{equation} \label{eq:0.1}
\alpha_{j}<\beta_{j}+\alpha_{j+1}+\beta'_{j} ,
\end{equation}
\begin{equation} \label{eq:0.2}
\alpha_{j}<\beta_{j-1}+\alpha_{j-1}+\beta'_{j-1},
\end{equation}
for $1 \leq j \leq k$.

Let $B_{j}$ (respectively, $B'_{j}$) be a path graph with endpoints $u_{j}$  and $v_{j}$ (respectively, $u'_{j}$  and $v'_{j}$)  and length $\beta_{j}$ (respectively, $\beta'_{j}$) for  $0\leq j \leq k$ (respectively, $1\leq j \leq k-1$).  Let $A_{j}$ be a path graph  with endpoints $a_{j}$  and $a'_{j}$  and length $\alpha_{j}$  for $1\leq j \leq k$.

If $A=(\alpha_{1},\dots,\alpha_{k} )$,  $B=(\beta_{0},\dots,\beta_{k} )$, $B'=(\beta'_{1},\dots,\beta'_{k-1} )$, then   we define $G_{A,B,B'}$ as the graph obtained from $A_{1},  \dots ,A_{k}$,  $B_{0},  \dots ,B_{k}$, $B'_{1},  \dots ,B'_{k-1}$ by identifying the vertices  $v_{j-1}$, $u_{j}$, and $a_{j}$ in a single vertex $p_{j}$ for each $1\leq j \leq k$, the vertices  $v'_{j-1}$, $u'_{j}$, and $a'_{j}$ in a single vertex $p'_{j}$ for each $1< j < k$,  the vertices  $u_{0}$, $u'_{1}$,$a'_{1}$ in a single vertex $p'_{1}$ and the vertices $v_{k-1}'$, $v_{k}$, $a'_{k}$ in a single vertex $p'_{k}$.
\end{defn}

Denote by $C_{j}$  the cycle in $G_{A,B,B'}$ induced by  $V(B_{j}) \cup V(B'_{j}) \cup V(A_{j}) \cup V(A_{j+1}) $, for $0\leq j \leq k$, where $V(B'_{0})=V(B'_{k})=V(A_{0})=V(A_{k+1})=\emptyset$. Note that $C_{j} \cap C_{j+1} =  A_{j+1} $ for every $0 \leq j \leq k-1$ and $C_{j} \cap C_{i} =  \emptyset $ if $\vert i-j \vert > 1$.

The following result is a direct consequence of inequalities  (\ref{eq:0.1})  and  (\ref{eq:0.2}).

\begin{lem} \label{Xx}
  If $x,y\in C_{j}$ with $0 \leq j \leq k$ and $\gamma$ is a geodesic in $G_{A,B,B'}$ joining $x$ and $y$, then $\gamma$ is contained in  $C_{j}$.

 \end{lem}





The following proposition gives an upper bound for the hyperbolicity constant of the graphs in this family.

\begin{prop} \label{general}

  $$\delta(G_{A,B,B'})\leq  \max{  \Big\lbrace  \smash{\displaystyle\max_{0 \leq j \leq k}}  { \dfrac{\beta_{j}+\beta'_{j}+ \max{ \lbrace 3\alpha_{j}+\alpha_{j+1}, \alpha_{j}+3\alpha_{j+1} \rbrace}  }{4}  } ,  \smash{\displaystyle\max_{0 < j < k}}  { \dfrac{\alpha_{j}+\alpha_{j+1} + \max{ \lbrace \beta_{j}, \beta'_{j}\rbrace}  }{2}  }    \Big\rbrace} .$$
 \end{prop}

\begin{proof}
In order to simplify notation, we shall write $G=G_{A,B,B'}$.

Theorem \ref{t:003} gives that there exists  some geodesic triangle $T=\lbrace x,y,z \rbrace\in \mathbb{T}_1$ and $p\in [xy]$ such that $\delta(G)=\delta(T)=d_{G}(p, [xz] \cup [yz])$.

Case (1).
Assume first that $T=C_{j}$ for some $0\leq j \leq k$. Thus, $\delta (G) =\delta(T)= L(C_{j})/4=( \beta_{j}+\beta'_{j}+ \alpha_{j}+\alpha_{j+1})/4$.

Case (2).
Assume now that $T$ is the closure of $(C_{i} \cup C_{i+1}\cup \dots \cup C_{i+r}) \smallsetminus ( A_{i+1} \cup A_{i+2}\cup \dots \cup A_{i+r})$, for some $0 \leq i < i+r\leq k$.

Case (2.1).
Assume $p\in C_{j}$, with $i < j <i+r $. Then, $p \in B_{j} \cup B'_{j} $.

Assume  $p\in B_{j}$. Since $i < j <i+r $, Lemma \ref{Xx} gives that $B'_{j}\subseteq [xz] \cup [yz]$. Thus, we conclude
$$\delta (G) =d_{G}(p, [xz] \cup [yz]) \leq  d_{G}(p, \{p'_{j},p'_{j+1}\})  \leq \dfrac{1}{2}( \beta_{j}+ \alpha_{j}+\alpha_{j+1}).$$

Similarly, if $p\in B'_{j}$, we conclude  $\delta (G) \leq ( \beta'_{j}+ \alpha_{j}+\alpha_{j+1})/2$.

Case (2.2).
Assume $p\in C_{i}$.

If  $[xy] \cap \lbrace p_{i+1},p'_{i+1} \rbrace = \emptyset$, then $ [xy] \subset C_{i} $ and $L([xy])\leq L(C_{i})/2$. Thus,
$$\delta (G)= d_{G}(p, [xz] \cup [yz]) \leq  d_{G}(p, \{x,y\}) \leq \dfrac{1}{2} L([xy])\leq  \dfrac{1}{4} L(C_{i})  \leq \dfrac{1}{4}( \beta_{i}+\beta'_{i}+ \alpha_{i}+\alpha_{i+1}).$$

If  $ p_{i+1} \in [xy]$, then by inequality ($\ref{eq:0.2}$) we have $p'_{i+1} \in [xz] \cup [yz]$ and $L([xy] \cap C_{i} )\leq L(C_{i})/2$.

Note that  $p \in [xy]\cap C_{i} \subset ([xy]\cap C_{i} )\cup A_{i+1}$,  $L(([xy]\cap C_{i} )\cup A_{i+1}) \leq L(C_{i})/2 + L(A_{i+1})$ and the endpoints of $([xy]\cap C_{i} )\cup A_{i+1}$ belong to $[xz] \cup [yz]$. Thus,
$$\delta (G) \leq   \dfrac{ L(C_{i})}{4}+   \frac{\alpha_{i+1}}{2} \leq   \frac{\beta_{i}+\beta'_{i}+ \alpha_{i}+\alpha_{i+1}}{4}  + \frac{\alpha_{i+1}}{2} =  \frac{\beta_{i}+\beta'_{i}+ \alpha_{i}+3\alpha_{i+1}}{4}. $$

Analogously, if  $ p'_{i+1} \in [xy]$, we obtain the same result.

Case (2.3).
Finally, if $p\in C_{i+r}$, then a similar argument to the one in (2.2) gives
$$\delta (G) \leq   \frac{\beta_{i+r}+\beta'_{i+r}+ 3\alpha_{i+r}+\alpha_{i+r+1}}{4}. $$

Since
$$ \beta_{0}+\beta'_{0}+\alpha_{0}+3\alpha_{1} = \beta_{0}+3\alpha_{1}=\beta_{0}+\beta'_{0}+ \max{ \lbrace 3\alpha_{0}+ \alpha_{1} , \alpha_{0}+ 3\alpha_{1} \rbrace}  ,$$
$$ \beta_{k}+\beta'_{k}+3\alpha_{k}+\alpha_{k+1} = \beta_{k}+3\alpha_{k}=\beta_{k}+\beta'_{k}+ \max{ \lbrace 3\alpha_{k}+ \alpha_{k+1} , \alpha_{k}+ 3\alpha_{k+1} \rbrace}  ,$$
we conclude
 $$\delta(G_{A,B,B'})\leq  \max{  \Big\lbrace  \smash{\displaystyle\max_{0 \leq j \leq k}}  { \dfrac{\beta_{j}+\beta'_{j}+ \max{ \lbrace 3\alpha_{j}+\alpha_{j+1}, \alpha_{j}+3\alpha_{j+1} \rbrace}  }{4}  } ,  \smash{\displaystyle\max_{0 < j < k}}  { \dfrac{\alpha_{j}+\alpha_{j+1} + \max{ \lbrace \beta_{j}, \beta'_{j}\rbrace}  }{2}  }    \Big\rbrace} $$
 in any case.
\end{proof}

Proposition \ref{general} has the following consequence.

\begin{cor} \label{CX}
If $\alpha_{j}=1$ for $0< j < k$, then
$$\delta(G_{A,B,B'})\leq    \max{ \bigg\lbrace \frac{1}{2} +\frac{1}{4}L(C_{0}), \frac{1}{2} +\frac{1}{4}L(C_{k}),\smash{\displaystyle\max_{0 < j < k}}  { \dfrac{2 + \max{ \lbrace \beta_{j}, \beta'_{j}\rbrace}  }{2}  } \bigg\rbrace} .       $$

 \end{cor}

\begin{proof}
We have $  \beta_{j}+\beta'_{j}+  \max{ \lbrace 3\alpha_{j}+ \alpha_{j+1} , \alpha_{j}+ 3\alpha_{j+1} \rbrace} =     \beta_{j}+\beta'_{j}+ \alpha_{j}+\alpha_{j+1}+ \max{ \lbrace 2\alpha_{j}, 2\alpha_{j+1} \rbrace} = L(C_{j}) +2 $, for $0\leq j\leq k$. Furthermore, if $0<j<k$, then    $  \beta_{j}+\beta'_{j}+  \max{ \lbrace 3\alpha_{j}+ \alpha_{j+1} , \alpha_{j}+ 3\alpha_{j+1} \rbrace} =\beta_{j}+\beta'_{j}+4 \leq 2(\max{ \lbrace \beta_{j}, \beta'_{j}\rbrace}+2) $.
\end{proof}


In what follows, we denote by  $C_{a_{1},a_{2},a_{3}}$ the graph with three paths  with lengths  $a_{1} \leq a_{2} \leq a_{3}$ joining  two fixed vertices.

The next corollary was proved in \cite[Theorem 12]{RSVV}. We provide here a simpler proof following a different approach.

\begin{cor} \label{c:2RSSV}
  $\delta(C_{a_{1},a_{2},a_{3}}) =(a_{3}+ \min \{a_{2},3a_{1}\}) /4 .$
 \end{cor}

\begin{proof}

Consider a graph $G_{A,B,B'}$ as in Definition \ref{defn1}, with $k=1$, $\alpha_{1}=a_{1}$, $\beta_{0}=a_{3}$  and $\beta_{1}=a_{2}$. Thus, $G_{A,B,B'}$ is the union of three paths $A_{0},B_{0},B_{1}$ joining $p_{1}$ and $p'_{1}$.

Since $\beta'_{0}=\beta'_{1}=\alpha_{0}=\alpha_{2}=0$ and $a_{1}< a_{2}\leq a_{3}$, we have $\alpha_{1} < \min{  \lbrace \beta_{0},\beta_{1} \rbrace  }$, and equations (\ref{eq:0.1}) and (\ref{eq:0.2}) hold. Thus, we can write  $G_{A,B,B'}=C_{a_{1},a_{2},a_{3}}$.

Assume first that  $3a_{1}\leq a_{2}$.

Let $T= \lbrace x,y,z \rbrace$  be the geodesic triangle which is  the closure of $(C_{0} \cup C_{1}) \smallsetminus A_{1}$, with $x\in B_{0} $, $y,z\in B_{1}$, $d(x,p_{1})=(a_{3}+a_{1})/2$, $d(y,p_{1})=a_{1}$ and $d(z,p_{1})=3a_{1}$.  Let  $p$ be the midpoint of $[xy]$. Then,
$$\delta(T)= d_{G}(p, [xz] \cup [yz])=d_{G}(p, \lbrace x,y\rbrace)=\dfrac{L([xy])}{2}=  \dfrac{1}{2}  \Big( \dfrac{a_{3} +a_{1}}{2} +a_{1}  \Big)  =   \dfrac{a_{3} +3a_{1}}{4}.$$

Therefore,  $\delta(C_{a_{1},a_{2},a_{3}}) \geq(a_{3}+ 3a_{1}) /4 $.



If $ a_{2} < 3a_{1}$, then let $T= \lbrace x,y,z \rbrace$  be the geodesic triangle which is  the closure of $(C_{0} \cup C_{1}) \smallsetminus A_{1}$, with $x\in B_{0} $, $y\in B_{1}$, $d(x,p_{1})=(a_{3}+a_{1})/2$, $d(y,p_{1})=(a_{2}-a_{1})/2 < a_{1}$ and $z=p'_{1}$.  Let  $p$ be the midpoint of $[xy]$. Then,
$$\delta(T)= d_{G}(p, [xz] \cup [yz])=d_{G}(p, \lbrace x,y\rbrace)=\dfrac{L([xy])}{2}=  \dfrac{1}{2}  \Big( \dfrac{a_{3} +a_{1}}{2} + \dfrac{a_{2} -a_{1}}{2}  \Big)  =   \dfrac{a_{3} +a_{2}}{4}.$$

Thus, $\delta(C_{a_{1},a_{2},a_{3}}) \geq(a_{3}+ \min \{a_{2},3a_{1}\}) /4 $ in both cases.

Let us prove the converse inequality. Assume first that $a_{1}< a_{2}$. Proposition \ref{general} gives $\delta(C_{a_{1},a_{2},a_{3}})\leq  \smash{\displaystyle\max}{ \lbrace a_{3}+  3a_{1}, a_{2}+3a_{1}    \rbrace} /4 = (a_{3}+3a_{1})/4$. On the other hand, Theorem \ref{t:1CRSV} gives $\delta(C_{a_{1},a_{2},a_{3}})\leq  c(C_{a_{1},a_{2},a_{3}}) /4 = (a_{3}+a_{2})/4$. Thus, we conclude $\delta(C_{a_{1},a_{2},a_{3}}) \leq(a_{3}+ \min \{a_{2},3a_{1}\}) /4 $.




Finally, assume that $a_{1}= a_{2}$. Thus, $\delta(C_{a_{1},a_{2},a_{3}})\leq  c(C_{a_{1},a_{2},a_{3}}) /4 = (a_{3}+a_{2})/4=(a_{3}+ \min \{a_{2},3a_{1}\}) /4 $.


Thus, we conclude  $\delta(C_{a_{1},a_{2},a_{3}}) =(a_{3}+ \min \{a_{2},3a_{1}\}) /4 .$
\end{proof}

\begin{cor} \label{X}
  $\delta(C_{a_{1},a_{2},a_{3}}) \leq(a_{3}+ 3a_{1}) /4 $.
 \end{cor}

\begin{lem} \label{Z}
 For every graph $G$, $ \diam (G)\leq2$ if and only if $d(v,e)\leq 1$ for every $v\in V(G)$  and  $e\in E(G)$.
 \end{lem}

\begin{proof}
Assume that $ \diam (G)\leq2$. Given $v\in V(G)$  and  $e\in E(G)$, if $p$ is the midpoint of $e$, then $d(v,p)\leq 3/2$, since $d(v,p) $ is an odd multiple of $1/2$ less than 2. Hence, $d(v,e) =d(v,p)-1/2 \leq1$.

Assume now that $d(v,e)\leq1$ for $v\in V(G)$  and  $e\in E(G)$.  Given $v,w\in V(G)$, choose  $e\in E(G)$ with $w\in e$; thus, $d(v,w)\leq d(v,e)+1 \leq 2$. If $v\in V(G)$ and $p$ is the midpoint of $e\in E(G)$, then $d(p,v)\leq d(v,e)+1/2 \leq 3/2$. Finally, consider $p,q$ midpoints of $e_{p},e_{q} \in E(G)$, respectively; if $v$ is a vertex of $e_{q}$, then $d(v,p)\leq  3/2$ and $d(p,q)\leq d(p,v)+d(v,q)\leq 3/2+1/2=2$. Hence, $ \diam (G)\leq2$.
\end{proof}






\begin{lem} \label{000}
The integers  $a_{1}:=n-c+1$, $a_{2}:=g+c-n-1$ and  $a_{3}:=n-g+1$ are the only real numbers satisfying

$(1)$ the following equations:
 \begin{equation} \label{eq:11}
a_{1}+a_{2}=  g,
\end{equation}
\begin{equation} \label{eq:22}
a_{2}+a_{3}=  c,
\end{equation}
\begin{equation} \label{eq:33}
a_{1}+a_{2}+a_{3}=  n+1.
\end{equation}

$(2)$ $a_{1} \leq a_{2}$ $\Leftrightarrow$ $n\leq c-1+g/2$.

$(3)$ $a_{2} \leq a_{3}$ $\Leftrightarrow$ $n\geq g-1+c/2$.

$(4)$ $a_{2}\leq 3 a_{1}$  $\Leftrightarrow$ $n\geq c-1 + g/4$.

$(5)$ $a_{1}\geq1$ and $ a_{2}\geq 2$  if $a_{1}=1$.
\end{lem}

\begin{proof}

Equations (\ref{eq:11}), (\ref{eq:22}), (\ref{eq:33}), and (5) follow directly.

Consider the system of linear equations in (1).  Since the coefficient matrix is non-sigular,  $a_{1}$, $a_{2}$ and  $a_{3}$ are the only real numbers satisfying (\ref{eq:11}), (\ref{eq:22}), and (\ref{eq:33}).

The condition $n\leq c+g/2 -1$ is equivalent to
$$ 2n+2\leq g+2c \qquad \Leftrightarrow \qquad n-c+1 \leq g+c-n-1  \qquad \Leftrightarrow \qquad  a_{1}\leq a_{2}.$$

On the other hand, the condition  $n\geq g-1+c/2$  is equivalent to
$$ c+2g-2\leq 2n \qquad \Leftrightarrow  \qquad  g+c-n-1 \leq n-g+1    \qquad \Leftrightarrow \qquad  a_{2}\leq a_{3}.  $$

Finally, the condition $n\geq c-1 +g/4 $ is equivalent to
$$  a_{1}+a_{2}+a_{3}-1\geq a_{2} +a_{3} +(a_{1}+a_{2})/4 -1   \qquad \Leftrightarrow \qquad a_{2}\leq 3a_{1}.  $$ \end{proof}

We say that the triplet $(g,c,n)$ is \emph{v-admissible} if $\mathcal{G}(c,g,n)$ is not the empty set.

\begin{lem} \label{lem2}
The triplet $(g,c,n)$ is v-admissible if and only if we have either $g=c\leq n$ or $g<c$ and $n\geq g-1+c/2 $.
\end{lem}

\begin{proof}
Assume that $(g,c,n)$ is v-admissible. If $g=c$, then there is nothing to prove.

Assume $g<c$ and consider any graph $G \in \mathcal{G}(c,g,n)$.  Denote by $C_{g}$ and $C_{c}$ two cycles in $G$ with lenghts $g$ and $c$, respectively. If there is no path $\eta$ joining two different vertices of  $C_{c}$ with $  \eta \not\subset C_{c}$, then $C_{c}\cap C_{g}$ contains at most a vertex and we conclude $n\geq g-1 +c>  g-1+c/2 $.

Assume now that such path $\eta$ exists. Without loss of generality we can assume that $\eta \cap C_{c}$ is exactly two vertices. Let $\{u,v\}=\eta \cap C_{c}$ and consider the two different paths $\eta_{1},\eta_{2}$ contained in $C_{c}$ and joining $u$ and $v$.

Define $b_{0}=L(\eta)$, $b_{1}=L(\eta_{1})$ and $b_{2}=L(\eta_{2})$. Thus, $b_{1}+b_{2}=c$, $b_{0}+b_{1}\geq g$ and $b_{0}+b_{2}\geq g$. Then $b_{0}\geq g-b_{1}$, $b_{0}\geq g-b_{2}$, $b_{0}\geq (2g-b_{1}-b_{2})/2=g-c/2$ and $n+1 \geq b_{0} +c \geq g-c/2 +c = g+c/2.$

Consider now positive integers  $g,c,n$  with either $g= c $ or $g<c$ and $n\geq g-1+c/2$.

If $g=c$, then let us define define $k:= n-g\geq 0$. Consider a graph $G_{0}$ isomorphic to the cycle graph $C_{g}$, and $k$ graphs $G_{i}$, $1\leq i \leq k$ isomorphic to the path graph $P_{2}$. Fix a vertex $v_{i}\in V(G_{i})$ for each $0\leq i \leq k$. Let $G$ be the graph obtained from $G_{0},G_{1},\dots, G_{k} $ by identifying $v_{0},v_{1},\dots, v_{k} $ in a single vertex. It is clear that $\vert V(G) \vert= g+k =n$, $g(G)=c(G)=c(G_{0})=g$. Thus, $G \in \mathcal{G}(g,g,n)$.

Consider now the case where $g<c$ and $n\geq g-1+c/2$.

Assume first that $g-1+c/2 \leq  n \leq c-1+g/2$.

Consider three natural numbers $a_{1}=n-c+1$, $a_{2}=g+c-n-1$ and  $a_{3}=n-g+1$. Lemma  \ref{000} gives   $a_{1}+a_{2}= g$, $a_{2}+a_{3}=  c$, $a_{1}+a_{2}+a_{3}= n+1$, and $  a_{1}\leq a_{2} \leq a_{3}$. Thus, $C_{a_{1},a_{2},a_{3}}\in \mathcal{G}(g,c,n)$ and $(g,c,n)$ is v-admissible.

Finally, assume that $n > c-1+g/2 $. Let us define $a_{1}=\lfloor g/2 \rfloor$, $a_{2}=g-a_{1}$, $a_{3}=c-a_{2}$, where $\lfloor t \rfloor$ denotes the lower integer part of $t$, i.e., the largest integer not greater than $t$.  Since $2a_{2}\leq g < c=a_{2}+a_{3}$, we have $a_{2}< a_{3}$; furthermore, $a_{1}\leq a_{2}$, and we can define $G_{0}:= C_{a_{1},a_{2},a_{3}} $. Note that $g(G_{0})=g$, $c(G_{0})=c$, and $\vert V(G_{0}) \vert= c+a_{1}-1$. Let $k:=n-(c-1+a_{1})> 0$ and consider  $k$ graphs $G_{i}$, $1\leq i \leq k$ isomorphic to the path graph $P_{2}$. Fix a vertex $v_{i}\in V(G_{i})$ for each $0\leq i \leq k$.  Let $G$ be the graph obtained from $G_{0},G_{1},\dots, G_{k} $ by identifying $v_{0},v_{1},\dots, v_{k} $ in a single vertex. It is clear that $\vert V(G) \vert=\vert V(G_{0}) \vert +k= c+a_{1}-1+n-(c-a_{1}-1)=n$, $c(G)=c(G_{0})=c$ and $g(G)=g(G_{0})=g$. Thus, $G \in \mathcal{G}(g,c,n)$ and $(g,c,n)$ is v-admissible.
\end{proof}

We say that a vertex $v$ in a graph $G$ is a \emph{cut-vertex} if $G\setminus \{v\}$ is not connected.
A graph is \emph{biconnected} if it does not contain cut-vertices.
Given a graph $G$, we say that a family of subgraphs $\{G_{s} \}_s$ of $G$ is a \emph{T-decomposition} of $G$ if $\cup_s G_{s} = G $ and $G_{s}\cap G_{r}$ is either a cut-vertex or the empty set for each $s\neq r$.
Every graph has a T-decomposition, as the following example shows.
Given any edge in $G$, let us consider the maximal two-connected
subgraph containing it: this is the well-known \emph{biconnected decomposition} of $G$.

The following result  appears in \cite [Theorem 3]{BRSV2}.
\begin{thm} \label{l:bermu}
  Let $ G $ be a graph and $\{G_{s}  \}  $ any T-decomposition of $G$. Then, $\delta(G) =\sup_{s} \delta (G_{s}) $.
 \end{thm}

\begin{lem} \label{Y}
  If $(g,c,n)$ is a v-admissible triplet  and  $n'$ is an integer with $n'\geq n$, then $(g,c,n')$ is a v-admissible triplet and
  $$   A(g,c,n') \leq A(g,c,n) \leq  B(g,c,n) \leq B(g,c,n').  $$
 \end{lem}

\begin{proof}

Lemma \ref{lem2} gives that $(g,c,n')$ is a v-admissible triplet.

If $n'=n$, Corollary  \ref{newnew} gives the desired result. Thus, assume that $n'>n $.

It suffices to prove that for each  $G_{0} \in  \mathcal{G}(g,c,n)$, there exists $G \in  \mathcal{G}(g,c,n')$ with $  \delta (G)= \delta (G_{0})$.

If $g=c$,  then Corollary \ref{newnew} implies $ A(c,c,n') = A(c,c,n) =B(c,c,n)=B(c,c,n')= c/4$.

Assume now that $g<c$ and consider a  graph $G_{0} \in \mathcal{G}(g,c,n)$ and  graphs $G_{i}$ isomorphic to $P_{2}$, $1\leq i \leq n'-n$. Fix vertices $u_{i}  \in V(G_{i})$, for $0\leq i \leq n'-n$. Denote by $G$ the graph obtained from $G_{0},G_{1},\dots, G_{n'-n} $ by identifying $u_{0},u_{1},\dots, u_{n'-n} $ in a single vertex $v$. Since $v$ is a cut-vertex, the graphs $G_{i}$, $0\leq i \leq n'-n$ are a T-decomposition of $G$ and Theorem \ref{l:bermu} implies $\delta(G) =\delta (G_{0})$.
\end{proof}

\begin{thm}\label{extraT}
Let $(g,c,n)$ be a v-admissible triplet and $r$ a positive integer. Consider  graphs $G_{1},G_{2}\in \mathcal{G}(g,c,n)$ with $m_{1},m_{2}$ edges, respectively, and such that $\delta (G_{1})=A(g,c,n)$ and $\delta (G_{2})=B(g,c,n)$. Then
$$A(rg,rc,n_{1})\leq r A(g,c,n) \leq r B(g,c,n)\leq B(rg,rc,n_{2}),$$
for every $n_{1}\geq n+(r-1)m_{1}$ and $n_{2}\geq n+(r-1)m_{2}$.
\end{thm}

\begin{proof}
Denote by $G_{1}^{(r)}$ the graph obtained from $G_{1}$ by replacing   each original edge with a path of  legth $r$. Thus, $\vert V(G_{1}^{(r)}) \vert = n+(r-1)m_{1}$, $g(G_{1}^{(r)})=rg$ and $c(G_{1}^{(r)})=rc$. It is clear that
$$A(rg,rc,n+(r-1)m_{1})\leq \delta (G_{1}^{(r)}) = r \delta (G_{1})= r A(g,c,n).$$

If $n_{1} \geq n+(r-1)m_{1}$, then Lemma \ref{Z} allows to conclude $A(rg,rc,n_{1})\leq A(rg,rc, n+(r-1)m_{1}) \leq r A(g,c,n)$.

Analogously, we have $B(rg,rc,n_{2})\geq r B(g,c,n)$.
\end{proof}

\begin{cor}\label{cextraT}
Let $(g,c,n)$ be a v-admissible triplet and $r$ a positive integer. Consider  graphs $G_{1},G_{2}\in \mathcal{G}(g,c,n)$ with $m_{1},m_{2}$ edges, respectively, and such that $\delta (G_{1})=A(g,c,n)=g/4$ and $\delta (G_{2})=B(g,c,n)=c/4$. Then
$A(rg,rc,n_{1})=rg/4$ for every $n_{1}\geq n+(r-1)m_{1}$ and $B(rg,rc,n_{2})=rc/4$  for every   $n_{2}\geq n+(r-1)m_{2}$.
\end{cor}

The argument in the proof of  Lemma \ref{000} gives the following result.

\begin{lem} \label{0002}
The integers  $a_{1}:=m-c$, $a_{2}:=g+c-m$ and  $a_{3}:=m-g$ are the only real numbers satisfying

$(1)$ the following equations:
 \begin{equation} \label{eq:11}
a_{1}+a_{2}=  g,
\end{equation}
\begin{equation} \label{eq:22}
a_{2}+a_{3}=  c,
\end{equation}
\begin{equation} \label{eq:33}
a_{1}+a_{2}+a_{3}=  m.
\end{equation}

$(2)$ $a_{1} \leq a_{2}$ $\Leftrightarrow$ $m\leq c+g/2$.

$(3)$ $a_{2} \leq a_{3}$ $\Leftrightarrow$ $m\geq g+c/2$.

$(4)$ $a_{2}\leq 3 a_{1}$  $\Leftrightarrow$ $m\geq c+ g/4$.

$(5)$ $a_{1}\geq1$ and $ a_{2}\geq 2$  if $a_{1}=1$.
\end{lem}

We say that the triplet $(g,c,m)$ is \emph{e-admissible} if  $\mathcal{H}(c,g,m)$ is not the empty set.

The argument in the proof of  Lemma \ref{lem2}, using Lemma \ref{0002} instead of Lemma \ref{000}, gives  the following result.

\begin{lem} \label{lem2e}
The triplet $(g,c,m)$ is e-admissible if and only if we have either $g=c\leq m$ or $g<c$ and $m\geq g+c/2 $.
\end{lem}

The arguments in the proofs of  Lemma \ref{Y} and  Theorem  \ref{extraT}, respectively, give the following results.

\begin{lem} \label{Ye}
  If $(g,c,m)$ is a e-admissible triplet  and  $m'$ is an integer with $m'\geq m$, then
  $$  \alpha(g,c,m') \leq \alpha(g,c,m) \leq  \beta(g,c,m) \leq \beta(g,c,m').  $$
 \end{lem}

\begin{thm}\label{al1}
 If $(g,c,m)$ is an e-admissible triplet  and  $r$ is a positive integer, then
$$\alpha(rg,rc,rm)\leq r \alpha(g,c,m) \leq r \beta(g,c,m)\leq \beta(rg,rc,rm).$$
\end{thm}

\begin{cor}\label{al2}
 Let $(g,c,m)$ be an   e-admissible triplet and $r$ a positive integer.  Consider  graphs $G_{1},G_{2}\in \mathcal{H}(g,c,m)$ such that $\delta (G_{1})=\alpha(g,c,m)=g/4$ and $\delta (G_{2})=\beta(g,c,m)=c/4$. Then, $\alpha(rg,rc,rm)=rg/4$ and $\beta(rg,rc,rm)=rc/4$.
\end{cor}

\section{Bounds for $A(g,c,n)$}

We will need the following results.

The next theorem is a well-known fact (see, e.g.,   \cite[Theorem 8]{RSVV} for a proof).

\begin{thm}\label{main}
Let $G$ be any graph. Then
$$ \delta(G)\leq \frac{1}{2}\diam (G) .$$
\end{thm}

\begin{defn}
Given a graph $G$ and its  biconnected decomposition $\{G_{n} \}$, we define the \emph{effective diameter} as
$$
\diameff V(G):= \sup_n \diam V(G_{n}), \qquad \diameff (G):= \sup_n
\diam (G_{n}).
$$
\end{defn}

Theorems \ref{l:bermu} and \ref{main} have the following consequence.

\begin{lem}\label{corollaryB}
Let $G$ be any graph. Then
$$ \delta(G)\leq \frac{1}{2}\diameff (G) .$$
\end{lem}

The following result characterizes the graphs with hyperbolicity constant $1$
(see \cite[Theorem 3]{BeRoRoSi}).

\begin{thm}\label{t:diametro}
Let $G$ be any graph.
Then $\d (G) = 1$ if and only if $\diameff (G) = 2$.
\end{thm}

The following theorems appears in \cite [Theorem 7]{MRSV} and \cite [Theorem 2.6]{BRS}, respectively.

\begin{thm} \label{7j05}
Let $G$ be any graph. If there exists a
cycle $C$ in G with length $ L(C) \geq 4$, then
$$ \delta(G) \geq \frac{1}{4} \min{  \lbrace \sigma \text{ is a cycle in } G \text{ with } L(\sigma)\geq 4}  \rbrace.$$
\end{thm}


\begin{thm} \label{omarpp}
For every graph $G$, $\delta(G)$  is a multiple of $1/4$.
\end{thm}

Let us start by computing $A(g,c,n)$ for $g=3$ and $g=4$.

\begin{thm} \label{calculo1}
  For any integers $3\leq c \leq n$ we have
  \begin{equation*}
  A(3,c,n)=\left\lbrace
  \begin{array}{l}
     3/4, \quad \text{if } c = 3, \\
     1,  \quad     \text{if } c>3.\\
  \end{array}
  \right.
\end{equation*}
 \end{thm}

\begin{proof}
If $g=c=3$, Corollary \ref{newnew} gives $A(3,3,n)=3/4$.

If $g=3$, $c\geq 4$ and $G\in \mathcal{G}(3,c,n) $, then Theorem \ref{7j05} gives $\delta(G)\geq1$. Thus, $A(3,c,n)\geq 1$.

Let us consider the complete graph with $c$ vertices $K_{c}$, and $n-c$ graphs $G_{1}, \dots ,G_{n-c} $ isomorphic to the path graph $P_{2}$. Fix $v_{0}\in V(K_{c})$ and $v_{j} \in G_{j}$ for $ 1\leq j \leq n-c $. Let $G_{0}$ be the graph obtained from $K_{c},G_{1}, \dots ,G_{n-c} $ by identifying the vertices  $v_{0},v_{1}, \dots ,v_{n-c} $ in a single vertex $v$. Thus, $G_{0}\in \mathcal{G}(3,c,n)$. Since $v$ is a cut vertex of $G_{0}$, $ \lbrace K_{c},G_{1}, \dots ,G_{n-c} \rbrace $ is the biconnected decomposition of $G_{0}$. We have $\delta(G_{1})= \dots =\delta(G_{n-c})=0$, and Theorem \ref{l:bermu} gives $\delta(G_{0})=\delta(K_{c})=1$. Since $A(3,c,n)\leq \delta(G_{0})=1$, we conclude  $A(3,c,n)= 1$.
\end{proof}

\begin{thm} \label{calculo2}
  For every v-admissible triplet $(4,c,n)$,
  \begin{equation*}
  A(4,c,n)=\left\lbrace
  \begin{array}{l}
    1,   \quad \text{if } c  \text{ is } \text{even}, \\
     5/4,  \quad     \text{if } c \text{ is } \text{odd}.\\
  \end{array}
  \right.
\end{equation*}
 \end{thm}

\begin{proof}
Corollary \ref{newnew} gives $A(4,c,n)\geq 1$.

Assume first that $c$ is even. Let $\Gamma_{c}$ be the graph defined by $V(\Gamma_{c})=\lbrace v_{1}, \dots  ,v_{c} \rbrace$ and $E(\Gamma_{c})= \lbrace  [v_{i},v_{j}] $ $ \vert $  $1\leq i,j \leq c$, $i+j$ is odd$\rbrace$. In particular, $[v_{1},v_{2}],       \dots   ,[v_{c-1},v_{c}],[v_{c},v_{1}] \in  E(\Gamma_{c})  $, $g(\Gamma_{c})=4$ and $c(\Gamma_{c})=c$. That is, we have $\Gamma_{c} \in  \mathcal{G}(4,c,c)$.

Let us prove that given two edges $e_{1},e_{2}\in E(\Gamma_{c})$, there exists a cycle $\sigma$ with $e_{1},e_{2}\subset \sigma$ and $L(\sigma)=4$. If $e_{1}=[v_{i_{1}},v_{i_{2}}]$ and $e_{2}=[v_{i_{2}},v_{i_{3}}]$, then $i_{1}+i_{2}$ and $i_{2}+i_{3}$ are odd, and thus, $i_{1}+i_{3}$ is even. Since $c\geq 4$, there exists $i_{4} \notin \lbrace i_{1},i_{2},i_{3} \rbrace$ such that $i_{2}+i_{4}$ is even. Hence, $i_{1}+i_{4}$ and $i_{3}+i_{4}$ are even, and the cycle $ [v_{i_{1}},v_{i_{2}}] \cup[v_{i_{2}},v_{i_{3}}] \cup [v_{i_{3}},v_{i_{4}}]  \cup [v_{i_{4}},v_{i_{1}}]$ contains $e_{1}$ and $e_{2}$.

Assume that $e_{1}=[v_{j_{1}},v_{j_{2}}]$ and $e_{2}=[v_{j_{3}},v_{j_{4}}]$, with $e_{1} \cap e_{2} = \emptyset$. Since $j_{3}+j_{4}$ is odd, we have that either   $j_{1}+j_{3}$ or $j_{1}+j_{4}$ is odd. By symmetry, we can assume that $j_{1}+j_{3}$ is odd. Thus, $j_{1}+j_{4}$ is even and $j_{2}+j_{4}$ is odd. Hence, $ [v_{j_{1}},v_{j_{2}}] \cup[v_{j_{2}},v_{j_{4}}] \cup [v_{j_{4}},v_{j_{3}}]  \cup [v_{j_{3}},v_{j_{1}}]$ is the required cycle.

Therefore, we conclude that $\diam V(\Gamma_{c})\leq 2$, since every two points in $\Gamma_{c}$ are contained in a cycle with length 4. Finally, Theorem \ref{main} gives $\delta (\Gamma_{c}) \leq 1$.

Thus, Lemma \ref{Y} gives $1\leq A(4,c,n) \leq A(4,c,c) \leq \delta(\Gamma_{c}) \leq 1$, and we deduce $A(4,c,n)= 1$.

Assume that $c$ is odd. Seeking for a contradiction, assume that $A(4,c,n)=1$, i.e., there exists $G\in \mathcal{G}(4,c,n) $ with $\delta(G)=1$. Let $C_{c}$ be a cycle in $G$ with $L(C_{c})=c$ and $G_{0}$ be the two-connected component of $G$ containing $C_{c}$. By Theorem \ref{t:diametro}, $\diam (G_{0})\leq 2$. Fix $v \in V(C_{c})$. By Lemma  \ref{Z}, we have $d_{G_{0}}(v,e)\leq 1$ for every $e \in E(G_{0})$.

Denote by $v,v_{2},\dots,v_{c} $ the vertices in $C_{c}$ such that $ [v,v_{2}], [v_{2},v_{3}],\dots,[v_{c},v] \subset C_{c}$. Since $d_{G_{0}}(v,e)\leq 1$ for every $e\in E(G_{0})$ and $g(G_{0})\geq g(G)=4$, we can prove inductively that $ [v,v_{2j}]\in E(G_{0})$ for every $1\leq j\leq (c-1)/2$ and $ [v,v_{2j+1}]\notin E(G_{0})$ for every $1\leq j\leq (c-1)/2$. In particular, if $j= (c-1)/2$, then $ [v,v_{c}]\notin E(G_{0})$, which is a contradiction. Hence, $A(4,c,n)>1$. By Theorem Theorem \ref{omarpp}, we have $A(4,c,n)\geq 5/4$.




Assume first that $c=5$. Then, Lemma \ref{lem2} gives that the tripet $(4,5,n)$ is v-admissible if and only if $n\geq 6$. Corollary \ref{newnew} and Lemma \ref{Y}
give $ 5/4\leq A(4,5,n)\leq A(4,5,6)\leq 5/4$. Thus, $A(4,5,n)= 5/4$ for every v-admissible triplet $(4,5,n)$.


Assume $c>5$. Consider the graph $\Gamma_{c-1}$  defined as before. Denote by $\Lambda_{c}$ the graph obtained from $\Gamma_{c-1}$ by replacing a fixed edge $e_{0} \in E(C_{c-1})$ by a path $\eta$ of length 2. Since $c\neq 5$, $\Lambda_{c}\in \mathcal{G}(4,c,c)$ and $(4,c,c)$ is v-admissible. The previous argument gives that any two points in  $\Lambda_{c}$ are contained in a cycle with length at most 5, and therefore, $ \diam (\Lambda_{c})\leq 5/2$.  Thus, Lemma \ref{main} implies $\delta (\Lambda_{c})\leq 5/4$. Lemma \ref{Y} gives $ 5/4\leq   A(4,c,n) \leq A(4,c,c)\leq \delta (\Lambda_{c})\leq 5/4 $, and we conclude $ A(4,c,n)= 5/4$.
\end{proof}

The next result provides good bounds for $A(g,c,n)$.

\begin{thm} \label{cotaA}
Let $(g,c,n)$ be a  v-admissible triplet with $g\geq 5$.

\begin{itemize}
  \item If  $2g-2  \leq c< 3g-4$ with $c=2g-2+s$ $(0\leq s \leq g-3)$,  then
$$ \frac{g}{4} \leq A(g,c,n)\leq   \dfrac{g+2+s}{4}   . $$

 \item If  $r$ is a positive integer,  $g$ is even  and   $c= 2g-2+r(g-2)$, then
  $$ \frac{g}{4} \leq A(g,c,n)\leq \dfrac{g+2}{4}  .$$

 \item If  $r$ is a positive integer, $g$ is odd  and   $2g-2+r(g-2) \leq  c \leq 2g-1+r(g-2)$, then
  $$ \frac{g}{4} \leq A(g,c,n)\leq \dfrac{g+3}{4}  .$$

\item If $r$ and $s$ are integers with  $r\geq1$, $ s \geq 0$, $g$ is even  and  $2g-2+r(g-2)< c \leq 2g-2+r(g-2)+2(r+1)(s+1)$, then

$$ \frac{g}{4} \leq A(g,c,n)\leq \dfrac{g+4+2s}{4} .$$

\item If $r$ and $s$ are integers with  $r\geq1$, $ s \geq 0$, $g$ is odd  and  $2g-1+r(g-2)< c \leq 2g-2+r(g-2)+2(r+1)(s+1)$, then

$$ \frac{g}{4} \leq A(g,c,n)\leq \dfrac{g+5+2s}{4} .$$


\end{itemize}


\end{thm}

\begin{proof}

Case 1.

Assume that  $2g-2  \leq c< 3g-4$ with $c=2g-2+s$ $(0\leq s \leq g-3)$.

Consider the graph $C_{a_{1},a_{2},a_{3}}$  with  $a_{1}=1$, $a_{2} = g-1$ and $a_{3} =g-1+s$. Note that $g(C_{a_{1},a_{2},a_{3}})=a_{1}+a_{2}=g$, $c(C_{a_{1},a_{2},a_{3}})=a_{2}+a_{3}=2g-2+s=c$ and thus, $C_{a_{1},a_{2},a_{3}}\in \mathcal{G}(g,c,c) $. Since $ g\geq5 $, Corollary \ref{c:2RSSV} gives $\delta(C_{a_{1},a_{2},a_{3}}) =(a_{3}+ \min \{a_{2},3a_{1}\}) /4 = (g-1+s + \min \{g-1,3\}) /4 = (g+2+s)/4.$ Thus,  Corollary  \ref{newnew} and Lemma \ref{Y} imply $g/4 \leq A(g,c,n)\leq A(g,c,c)\leq \delta(C_{a_{1},a_{2},a_{3}})\leq  (g+2+s)/4.$
\smallskip

Case 2.

Assume that $2g-2+r(g-2) \leq  c \leq 2g-1+r(g-2)$, with $r\geq 1$.

Since $r\geq1$, it follows that $c\geq 2g-2+r(g-2)\geq 3g-4$.

Consider a graph $G_{A,B,B'}$ as in Definition $\ref{defn1}$, with $k=r+1$, $\beta_{0} =g-1$ and $\alpha_{j}=1$ for $1 \leq j \leq k$.
\smallskip

Case 2.1.

If $g$ is even and $c=2g-2+r(g-2)$, then let  $\beta_{k}= g-1 $ and  $\beta_{j}=\beta'_{j}=g/2-1$ for $1 \leq j \leq k-1$ in the previous graph $G_{A,B,B'}$. Thus,
$$  \sum_{j=0}^{k} ( \beta_{j}+\beta'_{j})  = g-1  +2r(g/2 -1 )+ g-1  = 2g-2+r(g -2 )     = c .$$

Note that $g(G_{A,B,B'})=L(C_{0})=\alpha_{1}+\beta_{0}=g$, $c(G_{A,B,B'})=\sum_{j=0}^{k} ( \beta_{j}+\beta'_{j}) =c $ and thus, $G_{A,B,B'}\in \mathcal{G}(g,c,c) $.

Since $L(C_{j})= g$ for $ 0\leq j \leq k $ and $2+ \max{ \lbrace \beta_{j},\beta'_{j} \rbrace} =1+ g/2$ for $ 0<j<k $, Corollaries \ref{newnew} and \ref{CX} and Lemma \ref{Y}  give
$$ \frac{g}{4} \leq    A(g,c,n) \leq    A(g,c,c)   \leq  \delta(G_{A,B,B'})\leq    \max{ \bigg\lbrace  \frac{2+g}{4},{ \dfrac{2 + g }{4}  } \bigg\rbrace} = \dfrac{2 + g }{4}.$$
\smallskip

Case 2.2.

Similarly, if $g$ is odd, consider a graph $G_{A,B,B'}$ as before with  $g-1 \leq \beta_{k} \leq g$,  $\beta_{j}= (g-1)/2$ and  $\beta'_{j}=(g-3)/2$ for $1 \leq j \leq k-1$. Since
$$  g-1+r((g-1)/2 + (g-3)/2 )+ g-1      \leq  \sum_{j=0}^{k} ( \beta_{j}+\beta'_{j})  \leq g-1  +r((g-1)/2 + (g-3)/2 )+ g  ,  $$
$$  2g-2+r(g -2 )     \leq  \sum_{j=0}^{k} ( \beta_{j}+\beta'_{j})  \leq 2g-1+r(g -2 ),  $$
we can choose $ \beta_{k} $ with the additional property $  \sum_{j=0}^{k} ( \beta_{j}+\beta'_{j}) =c $.

Note that $g(G_{A,B,B'})=L(C_{0})=\alpha_{1}+\beta_{0}=g$, $c(G_{A,B,B'})=\sum_{j=0}^{k} ( \beta_{j}+\beta'_{j}) =c $ and thus, $G_{A,B,B'}\in \mathcal{G}(g,c,c) $.

Since  $L(C_{j})\leq g+1$ for $ 0\leq j \leq k $ and $2+ \max{ \lbrace \beta_{j},\beta'_{j} \rbrace} =(g+3)/2 $ for $ 0<j<k $,  Corollaries \ref{newnew} and \ref{CX} and Lemma \ref{Y}  give
$$ \frac{g}{4} \leq    A(g,c,n) \leq    A(g,c,c)   \leq \delta(G_{A,B,B'})\leq    \max{ \bigg\lbrace  \frac{2+g+1}{4},{ \dfrac{3 + g }{4}  } \bigg\rbrace} = \dfrac{3 + g }{4}.$$
\smallskip

Case 3.

Assume now that $r$ and $s$ are integers with $r\geq1$, $s \geq 0$, $g$ is even and  $2g-2+r(g-2) < c \leq 2g-2+r(g-2)+2(r+1)(s+1)$.

Consider a graph $G_{A,B,B'}$ as in Definition $\ref{defn1}$, with $k=r+1$, $\beta_{0}=g-1$, $g-1+2s< \beta_{k} \leq g-1+2(s+1)$ and  $\alpha_{j}=1$ for $1 \leq j \leq k$.

If $g$ is even, let $ g/2-1+s \leq \beta_{j},\beta'_{j}\leq g/2+s$ for $0 < j < k$. Since
$$  g-1+2r(g/2 -1+s )+ g-1+2s      <  \sum_{j=0}^{k} ( \beta_{j}+\beta'_{j})  \leq g -1 +2r(g/2 +s )+ g-1+2(s+1)  ,  $$
$$ 2g-2+r(g -2 ) \leq 2g-2+r(g -2 )+ 2(r+1)s      < \sum_{j=0}^{k} ( \beta_{j}+\beta'_{j})  \leq 2g-2+r(g -2 )+ 2(r+1)(s +1)   ,  $$
we can choose $ \beta_{j},\beta'_{j} $ with the additional property $  \sum_{j=0}^{k} ( \beta_{j}+\beta'_{j}) =c $.

Note that $g(G_{A,B,B'})=L(C_{0})=\alpha_{1}+\beta_{0}=g$, $c(G_{A,B,B'})=\sum_{j=0}^{k} ( \beta_{j}+\beta'_{j}) =c $ and thus, $G_{A,B,B'}\in \mathcal{G}(g,c,c) $.
Since $L(C_{j})\leq g+2s+2$ for $ 0 \leq j \leq k $ and $2+ \max{ \lbrace \beta_{j},\beta'_{j} \rbrace} \leq (g+4)/2+s $ for $ 0<j<k $, Corollaries \ref{newnew} and \ref{CX} and Lemma \ref{Y}  give
$$ \frac{g}{4} \leq    A(g,c,n) \leq    A(g,c,c) \leq  \delta(G_{A,B,B'})\leq    \max{ \bigg\lbrace  \frac{4+g+2s}{4},{ \dfrac{4+g +2s }{4}  } \bigg\rbrace} = \dfrac{g+4 +2s }{4} .$$
\smallskip


Case 4.

Assume now that $r$ and $s$ are integers with $r\geq1$, $s \geq 0$, $g$ is odd and  $2g-1+r(g-2) < c \leq 2g-2+r(g-2)+2(r+1)(s+1)$.




Consider a graph $G_{A,B,B'}$ as in Definition $\ref{defn1}$, with $k=r+1$, $\beta_{0}=g-1$, $g+2s< \beta_{k} \leq g-1+2(s+1)$, $\alpha_{j}=1$ for $1 \leq j \leq k$ and $ (g-1)/2+s \leq \beta_{j} \leq(g+1)/2+s$, $(g-3)/2+s \leq \beta'_{j}\leq (g-1)/2+s $ for $0 < j < k$. Since
$$  g-1+r((g-1)/2+s+(g-3)/2+s )+ g+2s      <  \sum_{j=0}^{k} ( \beta_{j}+\beta'_{j})  \leq g -1 +r(  (g+1)/2+s   +(g-1)/2+s )+ g-1+2(s+1)  ,  $$
$$ 2g-1+r(g -2 )  < \sum_{j=0}^{k} ( \beta_{j}+\beta'_{j})  \leq 2g-2+r(g -2 )+ 2(r+1)(s +1)   ,  $$
we can choose $ \beta_{j},\beta'_{j} $ with the additional property $  \sum_{j=0}^{k} ( \beta_{j}+\beta'_{j}) =c $.

Note that $g(G_{A,B,B'})=L(C_{0})=\alpha_{1}+\beta_{0}=g$, $c(G_{A,B,B'})=\sum_{j=0}^{k} ( \beta_{j}+\beta'_{j}) =c $ and thus, $G_{A,B,B'}\in \mathcal{G}(g,c,c) $.
Since $L(C_{j})\leq g+2s+2$ for $ 0 \leq j \leq k $  and $2+ \max{ \lbrace \beta_{j},\beta'_{j} \rbrace} \leq (g+5)/2+s $ for $ 0<j<k $, Corollaries \ref{newnew} and \ref{CX} and Lemma \ref{Y}  give
$$ \frac{g}{4} \leq    A(g,c,n) \leq    A(g,c,c) \leq  \delta(G_{A,B,B'})\leq    \max{ \bigg\lbrace  \frac{4+g+2s}{4},{ \dfrac{5+g +2s }{4}  } \bigg\rbrace} = \dfrac{g+5 +2s }{4} .$$
\end{proof}

The following result shows that the lower bound $g/4 \leq A(g,c,n)$  is attained for infinitely many v-admissible triplets.

\begin{prop} \label{Xx}
For any positive integer $u$, we have  $A(4u,6u,n)= g/4$ for every $n\geq 9u-3$.
\end{prop}

\begin{proof}
Consider a cycle graph $C_{6}$ with vertices $v_{1}, \dots ,v_{6}$ and the graph $G$ with $V(G)=V(C_{6})$ and $E(G)=E(C_{6})\cup \lbrace  [v_{1},v_{4}  ],   [v_{2},v_{5}  ] ,  [v_{3},v_{6}  ]  \rbrace$. Thus,   $G\in \mathcal{G}(4,6,6) $. One can check that $\diam (G)=1$, and Theorem  \ref{main} gives $\delta (G)\leq 1$. Hence,  Corollary \ref{newnew} implies  $ 1\leq A(4,6,6) \leq \delta (G)\leq 1$ and $A(4,6,6)=1$. Since $G$ has 9 edges, Corollary \ref{extraT} gives $A(4u,6u,n)= g/4$ for every $n\geq 6+(u-1
)9 = 9u-3$.
\end{proof}

We give now some bounds for $A(g,c,m)$ which do not depend on $r$ and $s$.

\begin{thm} \label{cotaA1}
Let $(g,c,n)$ be a v-admissible triplet with $g\geq5$.

\begin{itemize}
  \item If $c<3g-4$, then
$$\dfrac{g}{4} \leq A(g,c,n)\leq \dfrac{2g-1}{4}.$$

\item If $c=3g-4$, then
$$\dfrac{g}{4} \leq A(g,c,n)\leq \dfrac{g+2}{4}      \qquad \text{ if } g \text{ is } \text{even}, $$
$$\dfrac{g}{4} \leq  A(g,c,n)\leq \dfrac{g+3}{4}    \qquad   \text{ if } g \text{ is } \text{odd}. $$

  \item If $c>3g-4$, then
$$\dfrac{g}{4} \leq A(g,c,n)\leq \dfrac{3g+5}{8}      \qquad \text{ if } g \text{ is } \text{even}, $$
$$\dfrac{g}{4} \leq  A(g,c,n)\leq \dfrac{3g+7}{8}    \qquad   \text{ if } g \text{ is } \text{odd}. $$
\end{itemize}


\end{thm}

\begin{proof}
By Corollary \ref{newnew}, it suffices to prove the upper bounds.

Case 1.

If $c<2g-2$, then Corollary \ref{newnew} gives the result. If $2g-2 \leq c < 3g-4  $, then Theorem \ref{cotaA} gives $A(g,c,n) \leq (g+2+s)/4 \leq (2g-1)/4$.
\smallskip

Case 2.

If $c=3g-4$, then Theorem \ref{cotaA} with $r=1$ gives the inequalities.
\smallskip

Case 3.

Consider the case $c>3g-4$.

Let us define $ r:=\big\lceil \frac{c-2g+2}{g-2} \big\rceil -1 $. Therefore,

$$\frac{c-2g+2}{g-2} -1 \leq r <  \frac{c-2g+2}{g-2} , $$
$$ 2g-2+r(g-2)< c \leq 2g-2+ (r+1)(g-2).   $$

Since $c> 3g-4$, we have  $c-2g+2> g-2$  and  $r\geq1$.

Define now $ s:=\big\lceil \frac{c-2g+2-r(g-2)}{2(r+1)} \big\rceil -1 $. Thus,

$$\frac{c-2g+2-r(g-2)}{2(r+1)}-1 \leq s <  \frac{c-2g+2-r(g-2)}{2(r+1)}, $$
$$ 2g-2+r(g-2)+2(r+1)s < c \leq 2g-2+r(g-2)+2(r+1)(s+1) .   $$

Since $c> 2g-2+r(g-2)$, we have $s\geq0$.

Since $  2g-2+r(g-2)+2(r+1)s < c \leq 2g-2+ (r+1)(g-2)$, we have $2(r+1)s \leq g-3$ and thus, $2s \leq (g-3)/(r+1) \leq (g-3)/2$.

If $g$ is even, then Theorem  \ref{cotaA} gives
$$A(g,c,n)\leq \dfrac{g+4+2s}{4} \leq \dfrac{g+4+\frac{g-3}{2}}{4} \leq \dfrac{3g+5}{8}.$$

If $g$ is odd, then Theorem  \ref{cotaA} gives
$$A(g,c,n)\leq \dfrac{g+5+2s}{4} \leq \dfrac{g+5+\frac{g-3}{2}}{4} \leq \dfrac{3g+7}{8}.$$

\end{proof}

We can improve the bounds in Theorem \ref{cotaA1} when $c$ is large enough.

\begin{thm} \label{cotaA2}
Let $(g,c,n)$ be a v-admissible triplet with $g\geq 5$ and $2c \geq g^{2}-2g+4$.

\begin{itemize}
  \item If $g$ is even, then
$$\dfrac{g}{4} \leq A(g,c,n)\leq \dfrac{g+4}{4}.$$

\item If $g$ is odd, then

$$\dfrac{g}{4} \leq A(g,c,n)\leq \dfrac{g+5}{4}.$$
\end{itemize}

\end{thm}

\begin{proof}

The lower bounds are consequence of Corollary \ref{newnew}. Let us prove the upper bounds.


Let us define $ r:=\big\lfloor \frac{c-2g+2}{g-2} \big\rfloor $. We have

$$\frac{c-2g+2}{g-2} -1 < r \leq  \frac{c-2g+2}{g-2} , $$
$$ 2g-2+r(g-2)\leq c < 2g-2+ (r+1)(g-2)    ,$$
$$ 2g-2+r(g-2)\leq c \leq 2g-3+ (r+1)(g-2)    .$$

Assume first that $g=5$. Thus, $8+3r\leq c \leq 10+ 3r$. Inequality  $2c \geq g^{2}-2g+4$ gives $c\geq10$. Thus, it suffices to consider the case $10<g\leq 10+3r$.


If $c=10$, we have $c=10<11=3g-4$, and Theorem \ref{cotaA1} gives $ A(g,c,n)\leq 9/4 = (g+4)/4 < (g+5)/4.$


Assume that $g\geq6$.

If $c=2g-2+r(g-2)$, Theorem \ref{cotaA} gives $ A(g,c,n)\leq (g+2)/4 $ if $g$ is even and $ A(g,c,n)\leq (g+3)/4 $ if $g$ is odd.

If $c=2g-1+r(g-2)$, Theorem \ref{cotaA} gives $ A(g,c,n)\leq (g+3)/4 $ if $g$ is odd.

Thus, we consider  $ 2g-2+r(g-2)< c \leq 2g-3+ (r+1)(g-2)$ if $g$ is even  and $ 2g-1+r(g-2)< c \leq 2g-3+ (r+1)(g-2)$ if $g$ is odd.

Since $g\geq 6$, we have $(g^{2}-2g+4)/2 \geq 3g-4$. Thus, $c\geq 3g-4 $, which implies  $c-2g+2\geq g-2$  and  $r\geq1$. Hence, $r\geq 1$ for every $g\geq5$.

On the other hand, note that for any value of $g$  we have
$$2c\geq g^{2}-2g+4  \quad \Leftrightarrow \quad   2c-4g+4 \geq g^{2}-6g+8  \quad \Leftrightarrow \quad   \frac{c-2g+2}{g-2} \geq \frac{g-4}{2}  .$$

Thus, $ (g-4)/2 < r+1 $ and we obtain $2r > g-6 $. Therefore, $2r \geq g-5$. Note that
$$ 2r \geq g-5 \quad \Leftrightarrow \quad   2(r+1)\geq g-2-1  \quad \Leftrightarrow \quad 2g-2 +r (g-2) + 2(r+1) \geq 2g-2 +(r+1)(g-2)-1   .$$

Thus,
$$ c \leq 2g-3+ (r+1)(g-2) \leq  2g-2 +r (g-2) + 2(r+1)$$
and Theorem \ref{cotaA} gives
$$ A(g,c,n)\leq \dfrac{g+4}{4}      \qquad \text{ if } g \text{ is } \text{even}, $$
$$  A(g,c,n)\leq \dfrac{g+5}{4}    \qquad   \text{ if } g \text{ is } \text{odd}. $$

\end{proof}

\section{Computation of $B(g,c,n)$}

We compute in this Section the exact value of $B(g,c,n)$ for every v-admissible triplet. Let us tart with two lemmas.

\begin{lem} \label{t1}
If $(g,c,n)$ is v-admissible and we have either $g=c$ or $g<c$ and $n\geq c-1+g/4$, then $B(g,c,n)=c/4$.
\end{lem}

\begin{proof}
If $g=c$, then Corollary \ref{newnew} implies $B(g,c,n)=c/4$. Assume now $g<c$ and $n\geq c-1+g/4$.

If  $n\leq c-1+g/2$, then $\max \lbrace  g-1+c/2   ,c-1+g/4 \rbrace \leq n  \leq c-1+g/2$. Consider three natural numbers $a_{1}:=n-c+1$, $a_{2}:=g+c-n-1$, and $a_{3}:=n-g+1$.  Lemma  \ref{000} gives  $C_{a_{1}, a_{2} , a_{3}} \in \mathcal{G}(g,c,n)$, with $a_{2} \leq 3a_{1} $. Corollary \ref{c:2RSSV} gives $\delta(C_{a_{1}, a_{2} , a_{3}}) =(a_{3}+ \min \{a_{2},3a_{1}\} ) /4= (a_{3}+ a_{2} )/4=c/4$.  Thus, $B(g,c,n) \geq c/4 $, and Corollary \ref{newnew} implies $B(g,c,n)=c/4$.

If $n> c-1+g/2$, then Corollary \ref{newnew}  and  Lemma \ref{Y} give  the result (since $ \lceil g/4 \rceil  \leq g/2$, where $ \lceil t \rceil $ denotes the upper integer part of $t$, there exists an integer $n_{0}$ such that $\max \lbrace  g-1+c/2   ,c-1+g/4 \rbrace  \leq n_{0} \leq c-1+g/2$). Thus, we conclude $B(g,c,n)=c/4$ in any case.
\end{proof}

\begin{lem} \label{t2}
If $(g,c,n)$ is v-admissible, $g<c$ and $ n  < c-1+g/4$, then $B(g,c,n)= n+1-(g+3c)/4$.
\end{lem}

\begin{proof}

First, let us prove that $B(g,c,n) \geq n+1-(g+3c)/4$.

Consider three natural numbers $a_{1}:=n-c+1$, $a_{2}:=g+c-n-1$ and  $a_{3}:=n-g+1$. Since $g<c$, we have $ g-1+c/2 \leq n$ by Lemma \ref{lem2}. Since $ g-1+c/2 \leq  n < c-1+g/4 < c-1+g/2$, Lemma  \ref{000} gives  $a_{1}\leq a_{2} \leq a_{3}$ and  $C_{a_{1}, a_{2} , a_{3}} \in \mathcal{G}(g,c,n)$. By Lemma \ref{000}, $ n  < c-1+g/4$ is equivalent to $3a_{1} < a_{2}$. Corollary \ref{c:2RSSV} gives $\delta(C_{a_{1}, a_{2} , a_{3}}) =(a_{3}+ \min \{a_{2},3a_{1}\} ) /4= (a_{3}+ 3a_{1} )/4=n+1-(g+3c)/4$.  Thus, $B(g,c,n) \geq n+1-(g+3c)/4$.

Now, let us prove that $B(g,c,n) \leq n+1-(g+3c)/4$.

Consider any graph $G \in \mathcal{G}(g,c,n)$.

Let $T$ be any fixed geodesic triangle in  $ G$. Note that $ g\leq L(T)\leq c$.

Assume first that $L(T)=c$.

Let us denote by $C_{g}$ a cycle in $G$ with length $g$. Since $n  < c-1+g/4$, $E(T)\cap E(C_{g}) \neq \emptyset$. Let $\eta$ be a fixed connected component of  $C_{g} \smallsetminus T $.

Let us denote by $G_{0}$ the subgraph of $G$ such that $V(G_{0})= V(T )\cup V(\eta)$ and $E(G_{0})= E(T) \cup E(\eta)$. Since $L(T)=c$, we have $G_{0} \in \mathcal{G}(g_{0},c,n_{0})$, with $n_{0}:= \lvert V(G_{0}) \rvert \leq n$ and $ g_{0}:=g(G_{0})\geq g$.  Let us define  $a_{1}: = L(\eta)=n_{0}-c+1$ and consider the two curves $\eta_{2},\eta_{3}$ contained in $T$ joining the endpoints of $\eta$. By symmetry we can assume that $a_{2}: = L(\eta_{2})\leq a_{3}: = L(\eta_{3})$. Since $\eta \subset C_{g}$, we have $a_{1}=L(\eta)\leq a_{2}$  and the following equations hold: $a_{1}+a_{2}=g_{0}$, $a_{2}+a_{3}=c$, $a_{1}+a_{2}+a_{3}=n_{0}+1$. Thus, Lemma \ref{000} gives $a_{2}= g_{0}+c-n_{0}-1$ and $a_{3} = n_{0}-g_{0}+1$. Note that $G_{0}=C_{a_{1},a_{2},a_{3}}$.  Corollary \ref{X} gives
$$\delta(G_{0})  \leq  \dfrac{a_{3}+3a_{1}}{4}   =  \dfrac{n_{0}-g_{0}+1+3(n_{0}-c+1)}{4} = n_{0}+1-\dfrac{g_{0}+3c}{4} \leq n+1-\dfrac{g+3c}{4}.$$

Since $d_{G}(u,v) \leq d_{G_{0}} (u,v)$ for every $u,v\in G_{0}$, we have that any geodesic $\gamma$ in $G$ contained in $G_{0}$ is a geodesic in $G_{0}$, $T$ is also a geodesic triangle in $G_{0}$, and
$$ \delta_{G}(T)  \leq \delta_{G_{0}}(T) \leq  \delta (G_{0})  \leq n+1-\dfrac{g+3c}{4}.$$

Assume now that $L(T)<c$.

Let us denote by $C_{c}$ a cycle in $G$ with length $c$. Since $n  <  c-1+g/4\leq c-1+L(T)/4$, we have $E(T)\cap E(C_{c}) \neq \emptyset$.

Denote by $k\geq1$  the cardinality of the connected components of $T \smallsetminus  C_{c}$.

Let us denote by  $L_{1},\dots , L_{k} $ the lengths of the connected components  $\eta_{1},\eta_{2},\dots , \eta_{k} $ of
$T \smallsetminus  C_{c}$, respectively.  Denote by  $\eta'_{1},\eta'_{2},\dots , \eta'_{k} $ the connected components of $C_{c} \smallsetminus  T$ such that $\eta_{j}$ and $\eta'_{j}$ have the same endpoints for $1\leq j\leq k$. Let us denote by $L'_{1},\dots , L'_{k} $  the lengths of  $\eta'_{1},\eta'_{2},\dots , \eta'_{k} $, respectively. Since the length of the smallest cycle in $G$ is $g$, we have  $L'_{j} \geq g- L_{j}$, for $1\leq j \leq k$.

Since $c\leq n$, inequality $  n  < c-1+g/4$ implies $g>4$. Note that

$$  g> 4 \hspace{0.2cm} \text{and} \hspace{0.2cm} k\geq1    \quad \Rightarrow \quad  g-4 \leq k(g-4)            \quad \Rightarrow \quad       0 < 4 \sum\limits_{j=1}^k L_{j} + \sum\limits_{j=1}^k (g-2L_{j})+4 -4k-g     $$

$$ \Rightarrow \quad   0 <4 \sum\limits_{j=1}^k L_{j} + \sum\limits_{j=1}^k (L'_{j}- L_{j} ) +4 -4\sum\limits_{j=1}^k 1 -g  \quad \Leftrightarrow \quad  c- \sum\limits_{j=1}^k L'_{j}  + \sum\limits_{j=1}^k L_{j} < 4 \Big( c+ \sum\limits_{j=1}^k (L_{j}-1) \Big) + 4 -g-3c $$

$$ \Rightarrow \quad    L(T)< 4n+4-g-3c  \quad \Rightarrow \quad \delta (T) < n+1-\frac{g+3c}{4}.$$

Hence,  $\delta(G) \leq n+1-(g+3c)/4$. Thus, $B(g,c,n)\leq n+1-(g+3c)/4$ and we conclude  $B(g,c,n)=n+1-(g+3c)/4$.
\end{proof}

We can summarize Lemmas \ref{lem2}, \ref{t1} and \ref{t2} in the following result.

\begin{thm} \label{12345}
For any v-admissible triplet $(g,c,n)$, the value of $B(g,c,n)$ is as follows.

(1) If we have either $g=c$ or $g<c$ and $n\geq c-1+g/4$, then $B(g,c,n)=c/4$.

(2) If $g<c$ and $n< c-1+g/4$, then $B(g,c,n)=n+1-(g+3c)/4$.
\end{thm}

\section{Bounds for $\alpha(g,c,m)$}

Let us start with some bounds for $\alpha(g,c,m)$ similar to the ones in Theorem \ref{cotaA}.

\begin{thm} \label{cotaAe}
Let $(g,c,m)$ be an  e-admissible triplet.

\begin{itemize}
 \item If  $2g-2  \leq c< 3g-4$ with $c=2g-2+s$ $(0\leq s \leq g-3)$,  then
$$ \frac{g}{4} \leq \alpha(g,c,m)\leq   \dfrac{g+2+s}{4}   . $$

\item If  $m\geq c+r$, $r$ is a positive integer,  $g$ is even  and   $c= 2g-2+r(g-2)$, then
  $$ \frac{g}{4} \leq \alpha(g,c,m)\leq \dfrac{g+2}{4}  .$$

\item If $m\geq c+r$, $r$ is a positive integer, $g$ is odd  and   $2g-2+r(g-2) \leq  c \leq 2g-1+r(g-2)$, then
  $$ \frac{g}{4} \leq \alpha(g,c,m)\leq \dfrac{g+3}{4}  .$$

\item If $m\geq c+r$, $r$ and $s$ are integers with  $r\geq1$, $ s \geq 0$, $g$ is even  and  $2g-2+r(g-2)< c \leq 2g-2+r(g-2)+2(r+1)(s+1)$, then

$$ \frac{g}{4} \leq \alpha(g,c,m)\leq \dfrac{g+4+2s}{4} .$$

\item If $m\geq c+r$,  $r$ and $s$ are integers with  $r\geq1$, $ s \geq 0$, $g$ is odd  and  $2g-1+r(g-2)< c \leq 2g-2+r(g-2)+2(r+1)(s+1)$, then

$$ \frac{g}{4} \leq \alpha(g,c,m)\leq \dfrac{g+5+2s}{4} .$$

\end{itemize}
\end{thm}

\begin{proof}

Corollary   \ref{newnew} gives the lower bound. Let us prove the upper bounds.

Assume that  $2g-2  \leq c< 3g-4$ with $c=2g-2+s$ $(0\leq s \leq g-3)$. Since $g<2g-2\leq c$, we have $m \geq c+1$.

Assume that $2g-2\leq c < 3g-4$. Consider the graph $C_{a_{1},a_{2},a_{3}}$  with  $a_{1}=1$, $a_{2} = g-1$ and $a_{3} =g-1+s$. Note that $g(C_{a_{1},a_{2},a_{3}})=a_{1}+a_{2}=g$, $c(C_{a_{1},a_{2},a_{3}})=a_{2}+a_{3}=2g-2+s=c$, $m=c+1$ and thus, $C_{a_{1},a_{2},a_{3}}\in \mathcal{H}(g,c,c+1) $. Corollary \ref{c:2RSSV} gives $\delta(C_{a_{1},a_{2},a_{3}}) =(a_{3}+ \min \{a_{2},3a_{1}\}) /4 = (g-1+s + \min \{g-1,3\}) /4 \leq (g+2+s)/4.$

Thus, Lemma \ref{Ye} implies $ \alpha(g,c,m)\leq \alpha(g,c,c+1)\leq \delta(C_{a_{1},a_{2},a_{3}})\leq  (g+2+s)/4.$

The proof in the other cases follows the argument in the proof of Theorem \ref{cotaA}, since $m\geq c+r$.
\end{proof}

The following result provides bounds for $\alpha(g,c,m)$ when $m<c+r$ and we can not apply Theorem  \ref{cotaAe}.

\begin{thm} \label{cotaAX}
Let $(g,c,m)$ be an  e-admissible triplet.

\begin{itemize}

\item If   $c= g$, then
  $$ \alpha(g,g,m) =    \beta(g,g,m)      = \frac{g}{4} .$$

\item If   $m= c+1$, then
  $$ \alpha(g,c,c+1) =    \beta(g,g,c+1)   = \frac{c-g+1+\min \lbrace  3,g-1 \rbrace  }{4} .$$

\item If $m\geq c+u$, with $  2\leq u \leq   \frac{c-g+1}{g-1} $, then
$$ \frac{g}{4} \leq \alpha(g,c,m)\leq \dfrac{5}{4}+ \dfrac{1}{4}  \Big\lceil \dfrac{c-g+1}{u}  \Big\rceil  $$

\end{itemize}
\end{thm}

\begin{proof}

Corollary   \ref{newnew} gives the lower bound. Let us prove the upper bounds.

Case 1.

If $c=g$, then  Corollary  \ref{newnew} gives $ \alpha(g,g,m) =g/4$.

If $m=c$, then $g=c=m$ and we have proved $ \alpha(g,g,g) =g/4$. Thus, we can assume $c>g$ and $m \geq c+1$. Furthermore, Lemma \ref{lem2e} gives $m\geq g+c/2$.
\smallskip

Case 2.

Let $m=c+1$. As in Section 2, denote by $C_{a_{1},a_{2},a_{3}}$ the graph with three paths with lenghts $a_{1} \leq a_{2}\leq a_{3}$ joining two fixed vertices. Let $a_{1}=1$, $a_{2}= g-1$, $a_{3}=c-g+1$. Since $m\geq g+c/2$ and $m=c+1$, we have $g-1 \leq c-g+1$. Then, every graph $G\in \mathcal{H}(g,c,c+1) $ is isomorphic to $C_{1,g-1,c-g+1}$. Corollary \ref{c:2RSSV}  gives $\delta(G) = (c-g+1+\min \lbrace  3,g-1 \rbrace )/4 $ and thus, $\alpha(g,c,c+1) =  \beta(g,c,c+1)     =  (c-g+1+\min \lbrace  3,g-1 \rbrace )/4 $.
\smallskip

Case 3.

Let $m=c+u$,  with $  2\leq u \leq   (c-g+1) / (g-1) $.

Consider a graph $G_{A,B,B'}$ as in Definition \ref{defn1}, with $k=u$, $\alpha_{j}=1$ for $ 1\leq j\leq u$ and $\beta_{0}=g-1 $. For each $1\leq j\leq u-1$, choose $\beta_{j}, \beta_{j}'$ such that we either have
$\beta_{j}+\beta_{j}'=\lfloor \frac{c-g+1}{u}  \rfloor$ or $\beta_{j}+\beta_{j}'=\lceil \frac{c-g+1}{u}  \rceil$
, and   $\sum_{j=1}^{u}(\beta_{j}+\beta_{j}')=c-g+1$.  Hence,      $\sum_{j=0}^{u}(\beta_{j}+\beta_{j}')=c$.

Note that $L(C_{0})=g$,
$$ u \leq  \frac{c-g+1}{g-1}  \qquad \Rightarrow \qquad g \leq  \frac{c-g+1}{u}+1 \qquad \Rightarrow \qquad  g \leq \Big\lfloor \frac{c-g+1}{u} \Big\rfloor+1  \qquad \Rightarrow \qquad  g\leq \beta_{j}+\beta'_{j} +1 $$
and thus, $L(C_{j}) \geq g$ for every $0 \leq j \leq u$. Thus, $G_{A,B,B'}\in \mathcal{H}(g,c,c+u) $.

Since  $\beta_{j}+ \beta_{j}' \leq \lceil \frac{c-g+1}{u}  \rceil  $ for each  $1 \leq j \leq u-1$, it is possible to choose $\beta_{j}, \beta_{j}'$ with the additional condition   $\max{ \lbrace \beta_{j}, \beta_{j}' \rbrace}  \leq \frac{1}{2} ( \lceil \frac{c-g+1}{u}  \rceil  +1)$.

Corollary  \ref{CX}  gives $\delta(G_{A,B,B'})\leq    \max{ \lbrace \frac{1}{2}+\frac{g}{4},   \frac{3}{4}+ \frac{1}{4}\lceil \frac{c-g+1}{u}  \rceil ,  \frac{5}{4}+ \frac{1}{4}  \lceil \frac{c-g+1}{u}  \rceil  \rbrace} = \frac{5}{4}+ \frac{1}{4}  \lceil \frac{c-g+1}{u}  \rceil $. Thus,  $\alpha(g,c,c+u) \leq \delta(G_{A,B,B'})\leq   \frac{5}{4}+ \frac{1}{4} \lceil \frac{c-g+1}{u}  \rceil $.

Finally,  Lemma  \ref{Ye} gives the desired result.
\end{proof}

We improve now Theorem \ref{cotaAX} for the case $g=3$.

\begin{thm} \label{calculo1m}
Let $(3,c,m)$ be an  e-admissible triplet.

\begin{itemize}

\item If $c=3$, then
$$\alpha(3,3,m)=   \beta (3,3,m)= \dfrac{3}{4}. $$

\item If $c\geq 4$ and  $m=c+1$,  then
$$\alpha(3,c,c+1) =  \beta (3,c,c+1)=   \dfrac{c}{4}. $$

 \item If  $c\geq 4$ and         $m= c+u$, with $  2\leq u < \lfloor  \frac{c-2}{2}  \rfloor$, then
$$ 1 \leq \alpha(3,c,c+u)\leq \frac{5}{4}+ \frac{1}{4} \Big\lceil \frac{c-2}{u}  \Big\rceil  .$$

\item If   $c\geq 4$ and      $c+ \lfloor  \frac{c-2}{2}  \rfloor \leq m < \binom  {c} {2}$,  then
$$1 \leq \alpha(3,c,m) \leq \frac{3}{2} .$$

\item If   $c\geq 4$ and    $m\geq   \binom  {c} {2}$, then
$$\alpha(3,c,m) =1. $$







\end{itemize}

\end{thm}

\begin{proof}
Theorem  \ref{cotaAX} gives the first three items.

If $c\geq 4$, Theorem  \ref{7j05} gives $\alpha (3,c,m)\geq 1$.

Assume that $m=c+ \lfloor  \frac{c-2}{2}  \rfloor .$

If $c$ is even, then $ m= c+(c-2)/2  $. Consider a graph $G_{A,B,B'}$ as in Definition \ref{defn1}, with $k=(c-2)/2$, $\alpha_{j}=1$ for $ 1\leq j\leq (c-2)/2$,  $\beta_{0}=\beta_{(c-2)/2}=2$ and $\beta_{j}=\beta'_{j}=1$ for $ 1\leq j\leq (c-2)/2-1$. Corollary  \ref{CX}  gives $\delta(G_{A,B,B'})\leq    \max{ \lbrace 5/4, 3/2   \rbrace}= 3/2 .$ Note that $G_{A,B,B'} \in \mathcal{H}(3,c,c+(c-2)/2) $. If $m\geq c+(c-2)/2 $, then Lemma  \ref{Ye} implies  $\alpha(3,c,m) \leq \alpha(3,c,c+(c-2)/2) \leq \delta(G_{A,B,B'})\leq   3/2    $.

If $c$ is odd, then $m=c+(c-3)/2$. Consider a graph $G_{A,B,B'}$ as in Definition \ref{defn1}, with $k=(c-3)/2$, $\alpha_{j}=1$ for $ 1\leq j\leq (c-3)/2$,  $\beta_{j}=\beta_{j}'=1$ for $1\leq j\leq (c-3)/2-1$, $\beta_{0}=3 $ and $ \beta_{(c-3)/2}=2$. Corollary  \ref{CX}  gives $\delta(G_{A,B,B'})\leq    \max{ \lbrace  5/4, 3/2  \rbrace} = 3/2$. Note that $G_{A,B,B'} \in \mathcal{H}(3,c,c+(c-3)/2) $. If $m\geq c+(c-3)/2$, then Lemma  \ref{Ye} implies  $\alpha(3,c,m) \leq \alpha(3,c,c+(c-3)/2) \leq \delta(G_{A,B,B'}) \leq    3/2    $.
\smallskip

Assume now that  $m = \binom  {c} {2}$.  Let $G$ be the complte graph with $c$ vertices. Thus,   $G\in \mathcal{H}(3,c,\binom  {c} {2}) $, $\delta(G)=1$ and $\alpha(3,c,\binom  {c} {2})= 1$. Finally,  Lemma  \ref{Y} gives the desired result.
\end{proof}

\section{Computation of $\beta(g,c,m)$}

The arguments in the proofs of Lemmas \ref{t1} and \ref{t2}, but using Lemmas \ref{0002} and \ref{lem2e} instead of Lemmas \ref{000} and \ref{lem2},
respectively, give the following result.

\begin{thm} \label{12345e}
For any e-admissible triplet $(g,c,m)$, the value of $\beta(g,c,m)$ is as follows.

(1) If we have either $g=c$ or $g<c$ and $m\geq c+g/4$, then $\beta(g,c,m)=c/4$.

(2) If $g<c$ and $m< c+g/4$, then $\beta(g,c,m)=m-(g+3c)/4$.
\end{thm}

\end{document}